\documentclass[final]{siamltex}

\usepackage{amsmath}
\usepackage{amsfonts}
\usepackage{latexsym}
\usepackage{epsfig}
\usepackage{amssymb}
\usepackage{epstopdf}
\usepackage{mathenv}
\usepackage{cases}
\usepackage{eqnarray}

\def\addots{\mathinner{\mkern1mu\raise1pt\hbox{.}\mkern2mu\raise4pt\hbox{.}
     \mkern2mu\raise7pt\vbox{\kern7pt\hbox{.}}\mkern1mu}}

\newtheorem{remark}{{\it Remark}}[section]

\begin{document}

\title{On the functional equation $\displaystyle \alpha\bf{u}+\mathcal{C}\star(\chi \bf{u})=\bf{f}$}

\author{Philippe Ryckelynck\footnotemark[1]~~ Laurent Smoch\footnotemark[1]} 

\thanks{{ULCO, LMPA, F-62100 Calais, France.\hspace{1cm}e-mail: \{ryckelyn,smoch\}@lmpa.univ-littoral.fr\\
\indent \indent Univ Lille Nord de France, F-59000 Lille, France. CNRS, FR 2956, France.}}

\maketitle
\baselineskip15pt

\begin{abstract}
We study in this paper the functional equation 
\begin{center}
$\displaystyle \alpha \mathbf{u}(t)+\mathcal{C}\star(\chi \mathbf{u})(t)=\mathbf{f}(t)$
\end{center}
where $\alpha\in\mathbb{C}^{d\times d}$, $\mathbf{u},\mathbf{f}:\mathbb{R}\rightarrow\mathbb{C}^d$, $\mathbf{u}$ being unknown. 
The term $\mathcal{C}\star(\chi \mathbf{u})(t)$ denotes the discrete convolution of an almost zero matricial 
mapping $\mathcal{C}$ with discrete support together with the product of $\mathbf{u}$ and the characteristic 
function $\chi$ of a fixed segment. 
This equation combines some aspects of recurrence equations and/or delayed functional equations, so that we may 
construct a matricial based framework to solve it. 
We investigate existence, unicity and determination of the solution to this equation. 
In order to do this, we use some new results about linear independency of monomial words in matrix algebras.
\end{abstract}



\section{Introduction}

Let $[t_0,t_f]$ be some interval of time and $\chi$ its characteristic function, and $\varepsilon>0$ a fixed time delay. 
Let $(c_k)_{k\in\mathbb{Z}}$ be an almost zero sequence of matrices in $\mathbb{C}^{d\times d}$ where $d$ 
denotes the ``physical'' dimension. 
We define next $\mathcal{C}:\mathbb{R}\rightarrow\mathbb{C}^{d\times d}$ by $\mathcal{C}(t)=\left\{\begin{array}{ll}
c_{-t/\varepsilon} & \mbox{ if }t\in \varepsilon\mathbb{Z}\\
0 & \mbox{otherwise}
\end{array}\right.$. 
We denote by $I_d$ the identity matrix of size $d$ and we fix throughout a matrix $\alpha\in\mathbb{C}^{d\times d}$. 
We have the discrete convolution product over $\varepsilon\mathbb{Z}$:
\begin{center}
$\displaystyle (\mathbf{u}\star \mathbf{v})(t)=\sum_{s\in \varepsilon\mathbb{Z}}\mathbf{u}(s)\mathbf{v}(t-s)$, $t\in\mathbb{R}$
\end{center}
where $\mathbf{u}$ and $\mathbf{v}$ are two mappings defined in $\mathbb{R}$ with matrix values of format 
$(d\times d,d\times 1)$ or $(1\times d,d\times 1)$.\\
\\
This paper is devoted to the study of the following functional equation
\begin{equation}
\alpha \mathbf{u}+\mathcal{C}\star(\chi \mathbf{u})=\mathbf{f}
\label{eq1}
\end{equation}
which may be rewritten equivalently as
\begin{equation}
\alpha \mathbf{u}(t)+\sum_{k\in\mathbb{Z}}c_k\mathbf{u}(t+k\varepsilon)\chi(t+k\varepsilon)=\mathbf{f}(t),
\label{eq2}
\end{equation}
where $\mathbf{u},\mathbf{f}:\mathbb{R}\rightarrow\mathbb{C}^d$, $\mathbf{u}$ being unknown. 
The sum in the left hand-side of (\ref{eq2}) will be denoted throughout the paper by
 $\Box\mathbf{u}=\mathcal{C}\star(\chi \mathbf{u})$, where
the operator $\Box$ may occur in the context of discrete calculus of variations (see \cite{RS1,RS2}). 
Equation (\ref{eq1}) may be thought as a discrete version of the differential equation
 $\alpha\mathbf{u}(t)+\mathbf{u}'(t)=\mathbf{f}(t)$ where $\mathbf{u}'$ has been replaced by $\Box$. 
An analogous functional equation appears in the field of time-scale calculus and has led to many works, 
see for instance \cite{ABOP,Hi} and the literature therein. 
However, we do not require here that the operator $\Box$ behaves as a derivative.\\
\\
We define $N$ as the largest integer such that $c_N\neq 0$ or $c_{-N}\neq 0$. 
For example, when $N=1$, equation (\ref{eq2}) rewrites as
\begin{equation}
\mathbf{f}(t)-\alpha\mathbf{u}(t)=\left\{\begin{array}{ll}
0 & \mbox{ if }t<t_0-\varepsilon\\
c_1 \mathbf{u}(t+\varepsilon) & \mbox{ if }t_0-\varepsilon\leq t<t_0\\
c_1 \mathbf{u}(t+\varepsilon)+c_0\mathbf{u}(t) & \mbox{ if }t_0\leq t<t_0+\varepsilon\\
c_1 \mathbf{u}(t+\varepsilon)+c_0\mathbf{u}(t)+c_{-1}\mathbf{u}(t-\varepsilon) & \mbox{ if }t_0+\varepsilon\leq t\leq t_f-\varepsilon\\
c_0\mathbf{u}(t)+c_{-1}\mathbf{u}(t-\varepsilon) & \mbox{ if }t_f-\varepsilon<t\leq t_f\\
c_{-1}\mathbf{u}(t-\varepsilon) & \mbox{ if }t_f<t\leq t_f+\varepsilon\\
0 & \mbox{ if }t_f+\varepsilon<t
\end{array}\right..
\label{Nequal1}
\end{equation}
In general, the equations (\ref{eq2}) may be thought as a mixture between recurrence equations and delayed functional equations.

We focus in this paper on the following two problems. 
What are the analytical properties of functions of the shape $\Box\mathbf{u}$ and is the functional equation (\ref{eq1}) well-posed?

The paper is organized as follows. 
In Section 2 we provide some analytical properties of the functions $\Box\mathbf{u}$. 
Section 3 presents some technical lemmas for matricial linear recurrences. 
Section 4 is concerned with the non-vanishing property for entries and determinants of matricial polynomial functions. 
It ensures that the linear compatibility equations extracted when solving (\ref{eq1}) are Cramer for the specific case $N=1$. 
We prove in Sections 5 and 6 that (\ref{eq1}) is well-posed if $\alpha$ is invertible, with proper accuracies, according to 
the respective cases $N=1$ and arbitrary $N$. 
The case $\alpha=0$ is also investigated, and stands for the characterization of the range of the operator $\Box$. 
At last, Section 7 gives some concluding remarks and perspectives.

\section{Analytical properties of the functions $\Box\mathbf{u}$}

In order to comprehend equation (\ref{eq1}) and present some robust methods for solving it, we provide first some 
features of the operator $\Box$.
\begin{theorem}
Let $\nu\in\mathbb{N}\cup\{\infty\}$.
\begin{enumerate}
\item The operator $\Box$ maps the space of the functions $\mathbf{u}:\mathbb{R}\rightarrow \mathbb{C}^d$ of 
piecewise-$\mathcal{C}^\nu$ regularity, to the space of compactly supported functions
 $\mathbf{v}:\mathbb{R}\rightarrow \mathbb{C}^d$ of piecewise-$\mathcal{C}^\nu$ regularity, such that 
$\mbox{Supp}(\mathbf{v})\subset[t_0-N\varepsilon,t_f+N\varepsilon]$. 
\item If $\mathbf{u}:\mathbb{R}\rightarrow \mathbb{C}^d$ has $p$ points of discontinuity on $\mathbb{R}$, then $\Box\mathbf{u}$
 has at most $(4N+2)+(2N+1)p$ points of discontinuity on $\mathbb{R}$.
\item If $c_N$ or $c_{-N}$ is invertible then the kernel of $\Box$ consists of the functions which vanish on $[t_0,t_f]$. 
\item If $c_N$ or $c_{-N}$ is invertible then, for all $\mathbf{u}:\mathbb{R}\rightarrow \mathbb{C}^d$,
 $\Box\mathbf{u}$ is measurable, 
(respectively integrable, of piecewise-$\mathcal{C}^\nu$ regularity) on $\mathbb{R}$ if and only if $\mathbf{u}$ is measurable 
(respectively integrable, of piecewise-$\mathcal{C}^\nu$ regularity) on $[t_0,t_f]$.
\end{enumerate}
\end{theorem}
\begin{proof}
The main ingredient of the proof is the explicit formula for $\Box\mathbf{u}(t)$ provided in \cite{RS1}:
{\footnotesize \begin{equation}
\left\{\begin{array}{cc}
0 & \mbox{if }t<t_0-N\varepsilon\\
c_N\mathbf{u}(t+N\varepsilon) & \mbox{if }t_0-N\varepsilon\leq t<t_0-N\varepsilon+\varepsilon \\
c_{N-1}\mathbf{u}(t+N\varepsilon-\varepsilon)+c_N\mathbf{u}(t+N\varepsilon)& 
\mbox{if }t_0-N\varepsilon+\varepsilon\leq t<t_0-N\varepsilon+2\varepsilon \\
\vdots & \vdots \\
c_{N-p}\mathbf{u}(t+N\varepsilon-p\varepsilon)+\ldots+c_N\mathbf{u}(t+N\varepsilon) &
 \begin{array}{c}\mbox{if }p\in\{0,\ldots,2N-1\}\mbox{ and}\\t_0+(p-N)\varepsilon\leq t<t_0+(p+1-N)\varepsilon\end{array}\\
\vdots & \vdots\\
c_{-N}\mathbf{u}(t-N\varepsilon)+\ldots+c_N\mathbf{u}(t+N\varepsilon) & 
\mbox{if }t_0+N\varepsilon\leq t\leq t_f-N\varepsilon\\
\vdots & \vdots\\
c_{-N}\mathbf{u}(t-N\varepsilon)+\ldots+c_{p-N}\mathbf{u}(t-N\varepsilon+p\varepsilon) &
 \begin{array}{c}\mbox{if }p\in\{2N-1,\ldots,0\}\mbox{ and}\\t_f+(N-p-1)\varepsilon<t\leq t_f+(N-p)\varepsilon\end{array}\\
\vdots & \vdots\\
c_{-N}\mathbf{u}(t-N\varepsilon)+c_{1-N}\mathbf{u}(t-N\varepsilon+\varepsilon) &
 \mbox{if }t_f+N\varepsilon-2\varepsilon<t\leq t_f+N\varepsilon-\varepsilon \\
c_{-N}\mathbf{u}(t-N\varepsilon) & \mbox{if }t_f+N\varepsilon-\varepsilon<t\leq t_f+N\varepsilon\\
0 & \mbox{if }t_f+N\varepsilon<t
\end{array}
\right.
\label{CMA}
\end{equation}}
Assertion 1. is an obvious consequence of formula (\ref{CMA}). 
Next, let $\mathbf{u}:\mathbb{R}\rightarrow\mathbb{C}^d$, then formula (\ref{CMA}) shows that the $(4N+2)$ points
 $t_0+k\varepsilon$ and $t_f+k\varepsilon$, $|k|\leq N$ may be points of discontinuity of $\Box\mathbf{u}$. 
Moreover, if $t_{\delta}$ is any point of discontinuity of $\mathbf{u}$, the values $t_{\delta}+k\varepsilon$, $|k|\leq N$, give rise
 in general to points of discontinuity of $\Box\mathbf{u}$. 
Lastly, since there does not exist any other point of discontinuity of $\Box\mathbf{u}$, assertion 2. is proved. 
Let us prove assertion 3., i.e. the conditional injectivity of $\Box$. 
We first notice help to (\ref{CMA}) that if $\mathbf{u}$ vanishes in $[t_0,t_f]$, then all the terms occuring in
 (\ref{CMA}) are equal to 0 so that $\Box\mathbf{u}$ is zero everywhere. 
Conversely, suppose that $\det(c_N)\neq 0$. 
If $\mathbf{u}\in\ker(\Box)$, second row of formula (\ref{CMA}) shows that $\mathbf{u}=0$ in $[t_0,t_0+\varepsilon[$. 
Next, by using this result, the third row of (\ref{CMA}) shows that $\mathbf{u}=0$ in $[t_0+\varepsilon,t_0+2\varepsilon[$ and so on. 
Then we may prove easily that $\mathbf{u}=0$ in 
$[t_0,t_0+\varepsilon[\cup[t_0+\varepsilon,t_0+2\varepsilon[\cup\ldots\cup]t_f-\varepsilon,t_f]$. 
We may proceed similarly when $\det(c_{-N})\neq 0$ by starting from the last but one row of (\ref{CMA}). 
As a consequence, if $\mathbf{u},\mathbf{v}:\mathbb{R}\rightarrow \mathbb{C}^d$, we get 
\begin{center}
$\Box\mathbf{u}=\Box\mathbf{v}\Leftrightarrow\mathbf{u}=\mathbf{v}$ in $[t_0,t_f]$
\end{center}
which proves the third assertion. 
We proceed in the very same way to prove that measurability or integrability 
of $\mathbf{u}$ on $[t_0,t_f]$ is equivalent to the same property for $\Box\mathbf{u}$ on $\mathbb{R}$.
\end{proof}

\section{Preliminary lemmas for matricial linear recurrences}

As we shall see in Sections 5 and 6, solving (\ref{eq1}) queries to deal with matricial nonstationary recurrences of the shape
\begin{equation}
\mathbf{w}_{n+1}=\mathbf{M}_n\mathbf{w}_n+\mathbf{g}_n,~n\in\mathbb{N},
\label{rec_gen}
\end{equation}
where $\mathbf{M}_n$ is a $s\times s$ matrix and $\mathbf{w}_n,\mathbf{g}_n$ are vector 
sequences in $\mathbb{C}^s$, $s$ being a fixed integer. 
Since non-commutative products occur in the explicit formula, let us denote by 
\begin{center}
$\displaystyle \coprod_{k=0}^m\mathbf{M}_k=\mathbf{M}_m\mathbf{M}_{m-1}\ldots \mathbf{M}_0$
\end{center} 
this left-side product.
\begin{lemma}
The solution to the recurrence (\ref{rec_gen}) is given by
\begin{equation}
\mathbf{w}_n=\left(\coprod_{k=0}^{n-1} 
\mathbf{M}_k\right)\mathbf{w}_0+\sum_{k=0}^{n-1}\left(\coprod_{\ell=k+1}^{n-1}\mathbf{M}_\ell\right)\mathbf{g}_k.
\end{equation} 
\label{lem1-sagv}
\end{lemma}
\begin{proof}
The cases $n=0$ and $n=1$ are true and use the conventions of the empty sum and empty (left-)product. 
The proof for $n>1$ is straightforward by using induction. 
\end{proof}
The particular case when $\mathbf{M}$ is independent on $n$ is of some importance and gives rise to
 the solution to stationary matricial recurrence of the shape (\ref{rec_gen}) as follows:
\begin{equation}
\mathbf{w}_n=\mathbf{M}^{n}\mathbf{w}_{0}+\sum_{k=0}^{n-1}\mathbf{M}^{n-k-1}\mathbf{g}_{k}.
\label{sagv}
\end{equation} 

\noindent In the following two lemmas, we deal with sequences of non-commutative matricial polynomials.
\begin{lemma}
Let $d\geq 2$, $\beta,\gamma\in\mathcal{M}_{d,d}(\mathbb{C})$ and $(\delta_n')$ and $(\delta_n'')$ be 
the sequences of matrices defined by the recurrences
\begin{equation}
\delta_{n+1}'=\beta \delta_n'+\gamma \delta_{n-1}',~\delta_{n+1}''=\delta_n''\beta +\delta_{n-1}''\gamma,~n\geq 1
\label{recform}
\end{equation} 
with $\delta_0'=\delta_0''=0$ and $\delta_1'=\delta_1''=I_d$. 
Then, for all $n\geq 2$, we have 
\begin{equation}
\delta_n'=\delta_n''=\sum\beta^{m_1}\gamma^{m_2}\beta^{m_3}\gamma^{m_4}\ldots
\label{poly}
\end{equation}
where the sum is extended over all the multiplets $(m_1,m_2,\ldots,m_n)\in\mathbb{N}^n$ with 
\begin{equation}
m_j\geq 1\mbox{ if }2\leq j\leq n-1\mbox{ and }\sum_{i=1}^{\left[\frac{n}{2}\right]}(m_{2i-1}+2m_{2i})=n-1.
\label{multiplet}
\end{equation}
The number of monomials in the formula (\ref{poly}) is the Fibonacci number $\mathcal{F}_{n}$.
\label{lem2-sagv}
\end{lemma}
\begin{proof}
First, let us mention that formula (\ref{poly}) is true for $n=0$ and $n=1$ by using the empty sum convention and the initial conditions. 
Let us prove now the formula (\ref{poly}) for the sequence $(\delta_n')$. 
Suppose that the result (\ref{poly}) holds for $n$ and $n-1$ and that $\delta_{n-1}'$ and $\delta_n'$ are of lengths
 $\mathcal{F}_{n-1}$ and $\mathcal{F}_{n}$ respectively. 
Then, help to (\ref{poly}) and (\ref{recform}) we see that $\delta_{n+1}'$ is the sum of the words excerpted from the
 sum $\delta_n'$ multiplied on the left by $\beta$ and those from the sum $\delta_{n-1}'$ multiplied on the left by $\gamma$. 
So we have proved by induction that $\delta_n'$ has the shape 
\begin{equation*}
\delta_n'=\sum e_{(n,m_1,m_2,\ldots)}\beta^{m_1}\gamma^{m_2}\beta^{m_3}\gamma^{m_4}\ldots
\end{equation*}
where the coefficients $e_{(n,m_1,m_2,\ldots)}$ are convenient positive integers independent of $\beta$ and $\gamma$. 
Let us note that if $\beta=\gamma=I_d$ then, obviously, $\delta_n'=\mathcal{F}_n\times I_d$ so, the number of monomial
 words in the sum (\ref{poly}) is $\mathcal{F}_n$. 
Let us prove now that all the coefficients $e_{(n,m_1,m_2,\ldots)}$ are in fact equal to 1. 
Indeed, we see that all the $\mathcal{F}_{n-1}$ words coming from $\delta_{n-1}'$ are distinct from the 
$\mathcal{F}_{n}$ words arising from $\delta_n'$ since the first ones begin with $\beta$ while the second ones start with $\gamma$. 
All the words $\delta_{n+1}'$ occuring in (\ref{poly}) being distinct, their cardinality is 
$\mathcal{F}_{n-1}+\mathcal{F}_n=\mathcal{F}_{n+1}$. 
So far, we have proved the result (\ref{poly}) for $\delta_n'$ and the proof for $\delta_n''$ is the same.
\end{proof}
In the following, we set $\delta_n=\delta_n'=\delta_n''$, $\forall n\in\mathbb{N}$.
\begin{lemma}
Let $d\geq 2$, $\beta,\gamma\in\mathcal{M}_{d,d}(\mathbb{C}),(\mathbf{g}_n)_{n}\in(\mathbb{C}^d)^\mathbb{N}$, and let 
$(\mathbf{w}_n)_{n}\in(\mathbb{C}^d)\in\mathbb{N}$ 
the vector sequence satisfying for all $n\geq 1$ the following two-order matricial recurrence equation: 
\begin{equation}
\mathbf{w}_{n+1}=\beta \mathbf{w}_{n}+\gamma \mathbf{w}_{n-1}+\mathbf{g}_{n}.
\label{rec}
\end{equation}
Then for all $n\in\mathbb{N}^\star$ we have
\begin{equation}
\mathbf{w}_{n}=\delta_{n}\mathbf{w}_1+\delta_{n-1}\gamma\mathbf{w}_0+\sum_{k=1}^{n-1}\delta_{n-k}\mathbf{g}_{k}.
\label{rectn}
\end{equation} 
\label{lem4-sagv}
\end{lemma}
\begin{proof}
Let us define $\mathbf{M}=\begin{pmatrix}\beta & \gamma\\I_d & 0\end{pmatrix}$. 
We may easily prove by induction that $\mathbf{M}^n=\begin{pmatrix}\delta_{n+1} & \delta_{n}\gamma\\
\delta_{n} & \delta_{n-1}\gamma\end{pmatrix}$, for all $n\in\mathbb{N}^\star$, 
by using the recurrence (\ref{recform}) for the sequence $(\delta_n)$. 
Next, (\ref{rec}) may be rewritten as
\begin{center}
$\begin{pmatrix}\mathbf{w}_{n+1}\\\mathbf{w}_{n}\end{pmatrix}=
\mathbf{M}\begin{pmatrix}\mathbf{w}_{n}\\ \mathbf{w}_{n-1}\end{pmatrix}+\begin{pmatrix}\mathbf{g}_{n}\\0\end{pmatrix}$.
\end{center}
By setting $\mathbf{w}_n'=\begin{pmatrix}\mathbf{w}_{n}\\\mathbf{w}_{n-1}\end{pmatrix}$, 
$\mathbf{g}_n'=\begin{pmatrix}\mathbf{g}_{n}\\0\end{pmatrix}$ and by using a slightly adapted version of formula (\ref{sagv}),
 we get $\mathbf{w}_{n}=\mathbf{M}^{n-1}\mathbf{w}_1+\sum_{k=1}^{n-1}\mathbf{M}^{n-1-k}\mathbf{g}_k$. 
Now, the formula (\ref{rectn}) is obtained by reading the first row of the previous formula. 
\end{proof}

\section{Linear independency of words in matrix algebras}

We shall obtain a result which ensures that the foregoing linear compatiblity equations which arise when solving the 
functional equation (\ref{eq1}) are Cramer for almost all choice of matrices $\alpha,c_{-1},c_0,c_1\in\mathbb{C}^{d\times d}$ if $N=1$.

Let us call a \textit{matricial polynomial function} of two matricial indeterminates ($\zeta,\xi)$ an expression of the following shape
\begin{equation}
\psi(\zeta,\xi)=\sum \kappa_{m_1,m_2,\ldots}\zeta^{m_1}\xi^{m_2}\zeta^{m_3}\xi^{m_4}\ldots
\label{rdl}
\end{equation} 
with coefficients in $\mathbb{C}$. 
As an example, we may cite the sequence of matrices occuring in Lemma \ref{lem2-sagv}. 
Given such a function, for all $(p,q)\in\mathbb{N}^2$, we denote by $K_{pq}$ the sum of all the coefficients
 $\kappa_{m_1,m_2,\ldots}$ such that $m_1+m_3+\ldots=p$ and $m_2+m_4+\ldots=q$. 
Similarly, for all $(p,q,r,s)\in\mathbb{N}^4$, we introduce the sum $\tilde{K}_{p,q,r,s}$ of all the coefficients
 $m_1\times m_2\times\ldots\times\kappa_{m_1,m_2,\ldots}$ such that $m_1+m_3+\ldots=p$, $m_2+m_4+\ldots=q$, and where $r$ and $s$ count
 the occurences of $\zeta$ and $\xi$ respectively in the monomial of (\ref{rdl}) associated to the multiplets $(m_1,m_2,\ldots)$. 
We must have necessarily $|r-s|\leq 1$. 
We shall say that the function (\ref{rdl}) is \textit{generic} when at least one sum $K_{pq}$ and one sum $\tilde{K}_{p,q,r,s}$ 
are nonzero. 
For instance, the matrices $\delta_n$  occuring in (\ref{poly}) are the specializations at $(\beta,\gamma)$ of generic 
polynomials, since the coefficients are either 1 or 0.

Let us give now the main result of this section.
\begin{theorem}
Suppose we are given $l$ generic matricial polynomial functions $\psi^{(k)}(\zeta,\xi)$, and let us set 
$\psi^{(k)}_{i,j}(\zeta,\xi)=(\psi^{(k)}(\zeta,\xi))_{ij}$. 
Then, the set of couples $(\zeta,\xi)$ such that
\begin{center}
$\forall i,j,k$, $\psi^{(k)}_{i,j}(\zeta,\xi)\neq 0$, with $1\leq i,j\leq d$, $1\leq k\leq l$, and\\
$\forall k$, $\det(\psi^{(k)}(\zeta,\xi))\neq 0$, with $1\leq k\leq l$
\end{center}
consists in an open dense subset of $(\mathbb{C}^{d\times d})^2$.
\label{ilpmm}
\end{theorem}
\begin{proof}
In order to prove this result, we first show that if $\psi(\zeta,\xi)$ is the generic matricial polynomial 
function defined by (\ref{rdl}), then all its entries $\psi_{i,j}(\zeta,\xi)$ are nonzero polynomials w.r.t. their $2d^2$ indeterminates.
The proof goes as follows. 
Let us specialize $\psi$ at matrices of the shape
\begin{center}
$\zeta=\begin{pmatrix}x & z & 0_{1,d-2}\\
0 & x & 0_{1,d-2}\\
0_{d-2,1} & 0_{d-2,1} & I_{d-2}
\end{pmatrix}$ and $\xi=\begin{pmatrix}y & t & 0_{1,d-2}\\
0 & y & 0_{1,d-2}\\
0_{d-2,1} & 0_{d-2,1} & I_{d-2}
\end{pmatrix}$
\end{center}
for convenient $x,y,z,t\in\mathbb{C}$. 
We easily get 
\begin{equation}
\psi(\zeta,\xi)=\begin{pmatrix}\psi(x,y) & \tilde{\psi}(x,y,z,t) & 0_{1,d-2}\\
0 & \psi(x,y) & 0_{1,d-2}\\
0_{d-2,1} & 0_{d-2,1} & \psi(1,1)I_{d-2}\end{pmatrix}
\label{equation14}
\end{equation}
where $\psi$ is the function (\ref{rdl}) evaluated at complex numbers and 
is equal to $\sum K_{p,q}x^py^q$ while
 $\tilde{\psi}$ is the polynomial $\sum \tilde{K}_{p,q,r,s}x^{p-r}y^{q-s}z^rt^s$. 
Due to the hypotheses, the sets of couples $(x,y)$ or quadruplets $(x,y,z,t)$ such that $\psi(x,y)\neq 0$ or 
$\tilde{\psi}(x,y,z,t)\neq 0$ 
respectively, are open everywhere dense subsets in $\mathbb{C}^2$ and in $\mathbb{C}^4$ respectively. 
Now, in order to deal with the zero-like entries of 
(\ref{equation14}), 
we use the following property of $\psi$. 
We have
\begin{center}
$\psi(g^{-1}\zeta g,g^{-1}\xi g)=g^{-1}\psi(\zeta,\xi)g$,
\end{center} 
where $g$ is any invertible matrix in $\mathbb{C}^{d\times d}$. 
Especially, if we choose $g$ as a permutation matrix $\mu_\sigma$, we obtain 
\begin{center}
$\psi_{ij}(\mu_\sigma^{-1}\zeta \mu_\sigma,\mu_\sigma^{-1}\xi \mu_\sigma)=\psi_{\sigma(i),\sigma(j)}(\zeta,\xi)$.
\end{center} 
The symmetric group $\scriptstyle\mathfrak{S}_d$ acts transitively on the set of couples $(i,i)$ with $1\leq i\leq d$
 and transitively on the set of couples $(i,j)$ with $1\leq i,j\leq d$ and $i\neq j$. 
Thus, since $\psi_{1,1}$ and $\psi_{1,2}$ are nonzero polynomials, we may claim by using suitable specializations 
of $(\zeta,\xi)$ that every other entry $\psi_{i,j}$ of $\psi(\zeta,\xi)$ gives rise to a nonzero polynomial. 

Let us prove now that if $\psi$ is generic, then $\det(\psi(\zeta_{11},\ldots,\xi_{dd}))$ is a nonzero polynomial. 
Indeed, if we particularize $\zeta$ and $\xi$ to be two multiples of the identity matrix, i.e. 
$(\zeta,\xi)=(\rho I_d,\sigma I_d)$ with $\rho$ and $\sigma$ in $\mathbb{C}$, we get the polynomial
\begin{equation*}
\det(\psi(\rho I_d,\sigma I_d))=\left(\sum_{p,q}K_{p,q}\rho^{p}\sigma^{q}\right)^d=\psi(\rho,\sigma)^d
\end{equation*}
which is nonzero in $\mathbb{C}[\rho,\sigma]$ since there exists at leat one coefficient $K_{pq}\neq 0$.

Collecting all these informations, we may work with open dense sets in $\mathbb{C}^{2d^2}$ instead of dealing with nonzero polynomials. 
Let us consider for all $k\in\{1,\ldots,l\}$ and $i,j\in\{1,\ldots,d\}$ the closed algebraic varieties of codimension 1 
\begin{equation*}
\mathcal{F}_{ij}^{(k)}=\{(\zeta,\xi)\in(\mathbb{C}^{d\times d})^2/\psi^{(k)}_{ij}(\zeta_{11},\ldots,\xi_{dd})=0\}
\end{equation*}
and
\begin{center}
$\mathcal{F}_{\det}^{(k)}=\{(\zeta,\xi)\in(\mathbb{C}^{d\times d})^2/\det(\psi^{(k)}(\zeta_{11},\ldots,\xi_{dd}))=0\}.$
\end{center}
Since each hypersurface is meager, the union $\displaystyle (\cup_{ijk}\mathcal{F}^{(k)}_{ij})
\bigcup(\cup_{k}\mathcal{F}^{(k)}_{\det})$ is a closed meager subset of $(\mathbb{C}^{d\times d})^2$. 
Thus, its complementary is an everywhere open dense subset of $(\mathbb{C}^{d\times d})^2$. 
This ends the proof of the result. 
\end{proof}
\begin{remark}\rm
As an example, for all integer $m$, the set of couples $(\beta,\gamma)\in(\mathbb{C}^{d\times d})^2$ such that 
\begin{equation}
\det(\beta)\neq 0,~\det(\gamma)\neq 0,~\beta\gamma\neq\gamma\beta\mbox{ and }\det(\delta_n)\neq 0,~\forall n\in\{0,\ldots,m\},
\end{equation}
is an open dense subset of $(\mathbb{C}^{d\times d})^2$. 
In the very special case when the matrices $\beta,\gamma$ are multiples of $I_d$, i.e. 
$\beta=\tilde{\beta}I_d$ and $\gamma=\tilde{\gamma}I_d$ with $\tilde{\beta}$ and $\tilde{\gamma}$ in $\mathbb{C}$, 
we easily find help to (\ref{recform}) that 
\begin{center}
$\delta_n=\frac{1}{\sqrt{\tilde{\beta}^2+4\tilde{\gamma}}}\left(\left(\frac{\tilde{\beta}+\sqrt{\tilde{\beta}^2+
4\tilde{\gamma}}}{2}\right)^n-\left(\frac{\tilde{\beta}-\sqrt{\tilde{\beta}^2+4\tilde{\gamma}}}{2}\right)^n\right)I_d$.
\end{center}
The set of couples $(\tilde{\beta},\tilde{\gamma})$ such that $\frac{\tilde{\beta}+\sqrt{\tilde{\beta}^2+4\tilde{\gamma}}}
{\tilde{\beta}-\sqrt{\tilde{\beta}^2+4\tilde{\gamma}}}$ is not of the shape $\exp\left(\frac{2ik\pi}{n}\right)$ 
is open dense in $\mathbb{C}^2$ and, for those couples, we have $\det(\delta_{n})\neq 0$. 
\label{remark4.1}
\end{remark}

\begin{remark}\rm
The set of generic polynomials is an open dense subset in the vector space of matricial polynomial functions of given degree.
\label{remark4.2}
\end{remark}

\begin{remark}\rm
Let us mention how to proceed if we deal with non-generic polynomials. We may use in this case the 
so-called defect Theorem of P.M. Cohn, see \cite[Theorem 9.6.1, p. 283]{Lot}. If $\mathcal{K}$ is a commutative field and $\Omega$ 
is an alphabet of non commuting variables, we may denote by $\mathcal{K}\langle\langle \Omega\rangle\rangle$ the algebra of non 
commutative formal series in these variables. Then the defect theorem states that \textit{if $\varphi_1,\varphi_2$ are elements of 
$\mathcal{K}\langle\langle\Omega\rangle\rangle$ without constant term, satisfying a non-trivial relation $\Phi(\varphi_1,\varphi_2)=0$ 
for some non commutative series $\Phi$ in two variables then $\varphi_1,\varphi_2$ commute}. This theorem allows to exhibit 
at least one nonzero polynomial entry to $\psi(\zeta,\xi)$, but the conclusion is not as so accurate than in Theorem \ref{ilpmm}.
\label{remark4.3}
\end{remark}

\section{Solving explicitly the functional equation when $N=1$}

Let us show now how to solve the functional equation (\ref{eq1}) when $N=1$ by using the preceding results.
 We restrict ourselves to the quite opposite cases $\alpha$ invertible or $\alpha=0$. 
Let us define $M\geq 1$ as the integer part of $\frac{t_f-t_0}{\varepsilon}$.
\begin{theorem}
If $N=1$, for all quadruplets $(c_{-1},c_0,c_1,\alpha)$ in an open dense subset of $(\mathbb{C}^{d\times d})^{4}$, the equation
\begin{center}
$\displaystyle \alpha \mathbf{u}+\mathcal{C}\star(\chi \mathbf{u})=\mathbf{f}$
\end{center} 
admits one and only one solution.
\end{theorem}
\begin{proof}
We may assume by density that $c_1$ is invertible. If $M=1$, the result follows easily from the inspection of (\ref{Nequal1}).\\
The main part of the proof consists in solving (\ref{Nequal1}) and is valid for all $\alpha$. 
Outside $[t_0-\varepsilon,t_f+\varepsilon]$, one has $\mathbf{f}(t)=\alpha\mathbf{u}(t)$ which allows to 
determine $\mathbf{u}$ if $\alpha$ is invertible. Next, we focus on the interval $[t_0+\varepsilon,t_f-\varepsilon]$. 
By using (\ref{eq2}), we get the recurrence equation:
\begin{equation}
\mathbf{u}(\tau+(n+1)\varepsilon)=
c_1^{-1}(-(\alpha+c_0)\mathbf{u}(\tau+n\varepsilon)-c_{-1}\mathbf{u}(\tau+(n-1)\varepsilon)+\mathbf{f}(\tau+n\varepsilon)), 
\label{bound}
\end{equation}
$\forall \tau\in[t_0,t_0+\varepsilon[$ and $n\geq 1$. So we see that (\ref{bound}) is a recurrence of the shape (\ref{rec}). 
Using Lemma \ref{lem4-sagv} with 
\begin{equation}
\beta=-c_1^{-1}(\alpha+c_0),~~\gamma=-c_1^{-1}c_{-1},
\label{betagamma}
\end{equation}
and $\mathbf{w}_n=\mathbf{u}(\tau+n\varepsilon)$, $\mathbf{g}_n=c_1^{-1}\mathbf{f}(\tau+n\varepsilon)$ for $\tau$ fixed, we get 
\begin{equation}
\mathbf{u}(\tau+n\varepsilon)=\delta_{n}\mathbf{u}(\tau+\varepsilon)
+\delta_{n-1}\gamma\mathbf{u}(\tau)+\sum_{k=1}^{n-1}\delta_{n-k}c_1^{-1}\mathbf{f}(\tau+k\varepsilon),
\label{solsafe}
\end{equation}
for all instant $\tau+n\varepsilon\leq t_f$. Since $\tau\in[t_0,t_0+\varepsilon[$, we have in any case $n\leq M-1$. 
More explicitly, we shall use (\ref{solsafe}) in the following ranges of indices:
\begin{eqnarray}
\tau\in[t_0,t_f-M\varepsilon[\Rightarrow\max\{n/\tau+n\varepsilon\leq t_f\}=M\label{range1},\\
\tau\in[t_f-M\varepsilon,t_0+\varepsilon[\Rightarrow\max\{n/\tau+n\varepsilon\leq t_f\}=M-1\label{range2}.
\end{eqnarray}
We find it convenient to denote by $\ell=\ell(\tau)$ the maximal value of $n$ for which we may use (\ref{solsafe}), 
given by (\ref{range1}) or (\ref{range2}). Although the formula (\ref{solsafe}) is true for $n=1$, this result is tautological
 in this case. For $n\geq 2$, (\ref{solsafe}) expresses the solution to the functional equation (\ref{eq1}) in the interval 
$[t_0+2\varepsilon,t_f]$, as a linear combination of the restrictions $\mathbf{u}(\tau)$ and $\mathbf{u}(\tau+\varepsilon)$ of
 $\mathbf{u}$ to the two intervals 
$[t_0,t_0+\varepsilon[$ and $[t_0+\varepsilon,t_0+2\varepsilon[$. 
The coefficients of this linear relationship are independent on time and constitute the sequence of matrices 
$(\delta_n)$ which depend only on $\beta$ and $\gamma$.

Now, let us determine the two additional unknown functions $\mathbf{u}(\tau)$ and $\mathbf{u}(\tau+\varepsilon)$ 
by solving (\ref{eq2}) near the boundaries $t_0$ and 
$t_f$. In order to do this, let us write explicitly the four remaining equations:
\begin{eqnarray}
\mbox{if }t\in[t_0-\varepsilon,t_0[, & \mathbf{f}(t)=\alpha\mathbf{u}(t)+c_1\mathbf{u}(t+\varepsilon),\label{bound1}\\
\mbox{if }t\in[t_0,t_0+\varepsilon[, & \mathbf{f}(t)=(\alpha+c_0)\mathbf{u}(t)+c_1\mathbf{u}(t+\varepsilon),\label{bound2}\\
\mbox{if }t\in]t_f-\varepsilon,t_f], & \mathbf{f}(t)=(\alpha+c_0)\mathbf{u}(t)+c_{-1}\mathbf{u}(t-\varepsilon),\label{bound3}\\
\mbox{if }t\in]t_f,t_f+\varepsilon], & \mathbf{f}(t)=\alpha\mathbf{u}(t)+c_{-1}\mathbf{u}(t-\varepsilon)\label{bound4}.
\end{eqnarray}
We note that the restriction of $\mathbf{u}$ to $[t_0-\varepsilon,t_0[$ is undetermined. Let us call this function $\varphi$. 
Next we solve (\ref{bound1}) and (\ref{bound2}) as follows, by adapting slightly the range of the time variable. 
If $t\in[t_0,t_0+\varepsilon[$, (\ref{bound1}) provides
\begin{equation}
\mathbf{u}(t)=c_1^{-1}\mathbf{f}(t-\varepsilon)-c_1^{-1}\alpha\varphi(t-\varepsilon).
\label{solbound1}
\end{equation}
If $t\in[t_0+\varepsilon,t_0+2\varepsilon[$, (\ref{bound2}) and (\ref{solbound1}) give 
\begin{equation}
\mathbf{u}(t)=c_1^{-1}\mathbf{f}(t-\varepsilon)-c_1^{-1}
(\alpha+c_0)c_1^{-1}\mathbf{f}(t-2\varepsilon)+c_1^{-1}(\alpha+c_0)c_1^{-1}\alpha\varphi(t-2\varepsilon).
\label{solbound2}
\end{equation}
The formulas (\ref{solbound1}) and (\ref{solbound2}) give the two additional restrictions of $\mathbf{u}$ needed in the expansion 
(\ref{solsafe}). We note that these restrictions, at the time being, depend on the auxiliary restriction $\varphi$ of $\mathbf{u}$. 

Now it remains to solve (\ref{bound3}) and (\ref{bound4}). First, the two functions occuring in the r.h.s. of the equation 
(\ref{bound3}) are already known and may be expressed through (\ref{solsafe}). The equation (\ref{bound3}) may be thought as a 
constraint either on $\varphi$ or on the l.h.s. $\mathbf{f}$ of (\ref{eq1}). We may observe that $\varphi(\tau-\varepsilon)$ occurs 
two times in 
$\mathbf{u}(t)$ and two times in $\mathbf{u}(t-\varepsilon)$. So, the coefficient we are looking for is equal to
\begin{center}
$\theta=((\alpha+c_0)\delta_{\ell}+c_{-1}\delta_{\ell-1})\mbox{coeff of }\varphi\mbox{ in }\mathbf{u}(\tau+\varepsilon)]$\\
$+((\alpha+c_0)\delta_{\ell-1}+c_{-1}\delta_{\ell-2})\gamma[\mbox{coeff of }\varphi\mbox{ in }\mathbf{u}(\tau)]$,
\end{center}
where $\ell=M$ or $\ell=M-1$, depending on the location of $\tau$ w.r.t. $t_f-M\varepsilon$.
Direct inspection of these coefficients from (\ref{solbound1}) and (\ref{solbound2}) yields the explicit expression 
of the coefficient of $\varphi$ in (\ref{bound3}):
\begin{eqnarray}
\theta=((\alpha+c_0)\delta_{\ell}+c_{-1}\delta_{\ell-1})(c_1^{-1}(\alpha+c_0)c_1^{-1}\alpha)\nonumber\\
+((\alpha+c_0)\delta_{\ell-1}+c_{-1}\delta_{\ell-2})\gamma(-c_1^{-1}\alpha)\hskip 1cm
\end{eqnarray}
\begin{center}
$ = [((\alpha+c_0)\delta_{\ell}+c_{-1}\delta_{\ell-1})c_1^{-1}(\alpha+c_0)-
((\alpha+c_0)\delta_{\ell-1}+c_{-1}\delta_{\ell-2})\gamma]c_1^{-1}\alpha$.
\end{center}
Now, in order to simplify the matrix $\theta$, we use the formulas (\ref{betagamma}) to eliminate $\alpha+c_0$ and $c_{-1}$ and we get
\begin{center}
$\theta=c_1((\beta\delta_{\ell}+\gamma\delta_{\ell-1})\beta+(\beta\delta_{\ell-1}+\gamma\delta_{\ell-2})\gamma)c_1^{-1}\alpha$.
\end{center}
By using three times the recurrence (\ref{recform}), we obtain 
\begin{center}
$\theta=c_1\delta_{\ell+2}c_1^{-1}\alpha$.
\end{center}
At the time being, no assumption has been made about $\alpha$. From now on, we consider the case when $\alpha$ is invertible and 
chosen in such a way that $\det(\delta_{\ell+2})\neq 0$ (see (\ref{betagamma}) and Remark \ref{remark4.1}). 
Therefore, $\det(\theta)\neq 0$. Hence, (\ref{bound3}) may be solved w.r.t. 
$\varphi(\tau-\varepsilon)$. As a by-product of the previous results, $\varphi(\tau-\varepsilon)$ is a linear combination 
of the values of $\mathbf{f}$ on the set $\{\tau+k\varepsilon\,|k|\leq N\}$. It remains to solve equation (\ref{bound4}) on 
$]t_f,t_f+\varepsilon]$.  We get simply $\mathbf{u}(t)=\alpha^{-1}(\mathbf{f}(t)-c_{-1}\mathbf{u}(t-\varepsilon))$ for all
 $t\in]t_f,t_f+\varepsilon]$, where the r.h.s. has already been determined. Then we have proved that, for a generic quadruplet 
$(c_{-1},c_0,c_1,\alpha)$ and for all 
$\mathbf{f}:\mathbb{R}\rightarrow\mathbb{C}^d$, there exists one and only one solution to the functional equation (\ref{eq1}). 
\end{proof}
Next, let us consider the case when $\alpha$ is zero, which stands for the characterization of $\mbox{Im}(\Box)$. 
\begin{theorem}
For all triplets $(c_{-1},c_0,c_1)$ in an open dense subset of $(\mathbb{C}^{d\times d})^{3}$ and for all 
$\mathbf{g}:[t_0,t_f]\rightarrow\mathbb{C}^d$, there exists one and only one extension $\mathbf{f}$ of $\mathbf{g}$ to 
$\mathbb{R}$ such that $Supp(\mathbf{f})\subset[t_0-\varepsilon,t_f+\varepsilon]$ and the functional equation 
$\mathbf{f}=\Box\mathbf{u}$ admits one and only one solution in $[t_0,t_f]$.
\end{theorem} 
\begin{proof}
We may assume by density that $c_1$ and $c_{-1}$ are invertible. Let be given a function 
$\mathbf{g}:[t_0,t_f]\rightarrow\mathbb{C}^{d}$ and let us construct an appropriate extension 
$\mathbf{f}:\mathbb{R}\rightarrow\mathbb{C}^{d}$ of $\mathbf{g}$ as follows. We set 
\begin{center}
$\mathbf{f}(t)=\left\{\begin{array}{ll}
0 & \mbox{ if }t<t_0-\varepsilon\\
\mathbf{p}(t) & \mbox{ if }t_0-\varepsilon\leq t<t_0\\
\mathbf{g}(t) & \mbox{ if }t_0\leq t\leq t_f\\
\mathbf{q}(t) & \mbox{ if }t_f<t\leq t_f+\varepsilon\\
0 & \mbox{ if }t>t_f+\varepsilon
\end{array}\right.$
\end{center}
for some auxiliary functions $\mathbf{p}$ and $\mathbf{q}$ with values in $\mathbb{C}^d$. The proof consists in showing that
 $\mathbf{f}$ lies in the range of $\Box$ if and only if $\mathbf{p}$ and $\mathbf{q}$ are uniquely determined by $\mathbf{g}$. 
We keep the notations of the previous proof. When $\alpha=0$, the formulas (\ref{bound}) to (\ref{solbound2}) still hold, we
 have $\theta=0$ and the auxiliary function $\varphi$ does not occur anymore. Let us solve the two equations (\ref{bound3}) 
and (\ref{bound4}) w.r.t. the two restrictions $\mathbf{p}$ and $\mathbf{q}$ of $\mathbf{f}$. 
To write these equations, we introduce the operators 
\begin{center}
$\displaystyle \Delta_{\ell}\mathbf{f}(\tau)=\sum_{k=1}^{\ell-1}\delta_{\ell-k}c_1^{-1}\mathbf{f}(\tau+k\varepsilon)$.
\end{center} 
Let us notice that these operators may be thought as discrete integral operators. Plugging formula (\ref{solsafe}),
 with the maximal value $\ell$ allowed for $n$, in the equation (\ref{bound3}), we get for all $t\in]t_f-\varepsilon,t_f]$
\begin{center}
$\begin{array}{rcl}
\mathbf{f}(t) & = & 
c_0(\delta_{\ell}\mathbf{u}(\tau+\varepsilon)+\delta_{\ell-1}\gamma\mathbf{u}(\tau)+\Delta_{\ell}\mathbf{f}(\tau))+\\
 & & c_{-1}(\delta_{\ell-1}\mathbf{u}(\tau+\varepsilon)+\delta_{\ell-2}\gamma\mathbf{u}(\tau)+\Delta_{\ell-1}\mathbf{f}(\tau)),
\end{array}$
\end{center}
$t$ and $\tau$ being connected as previously by the requirements that $\tau\in[t_0,t_0+\varepsilon[$ and
 $\frac{t-\tau}{\varepsilon}\in\mathbb{N}$. Remembering the values of $\beta,\gamma$ given by (\ref{betagamma})
 and the recurrence (\ref{recform}), we get:
\begin{equation}
\mathbf{f}(t)=-c_{1}(\delta_{\ell+1}\mathbf{u}(\tau+\varepsilon)+
\delta_\ell\gamma\mathbf{u}(\tau))+c_0\Delta_\ell\mathbf{f}(\tau)+c_{-1}\Delta_{\ell-1}\mathbf{f}(\tau).
\label{ft}
\end{equation}
Since we may respectively rewrite (\ref{solbound1}) and (\ref{solbound2}) as
\begin{center}
$\mathbf{u}(\tau)=c_1^{-1}\mathbf{p}(\tau)$ and 
$\mathbf{u}(\tau+\varepsilon)=c_1^{-1}\mathbf{f}(\tau)-c_1^{-1}c_0c_1^{-1}\mathbf{p}(\tau)$,
\end{center}
we see that the coefficient of $\mathbf{p}=\mathbf{f}(\tau-\varepsilon)$ in (\ref{ft}) is given by:
\begin{center}
$(\delta_{\ell+1}c_1^{-1}c_0-\delta_\ell\gamma)c_1^{-1}=-\delta_{\ell+2}c_1^{-1}$.
\end{center}
We assume the triplet $(c_{-1},c_0,c_1)$ lies in the open dense subset where $\det(\delta_{\ell+2})\neq 0$ and thus,
 we have shown that the restriction $\mathbf{p}$ exists and is unique. This being done, the equation (\ref{bound3}) yields 
\begin{center}
$\mathbf{f}(t+\varepsilon)=\mathbf{q}(t)=c_{-1}(\delta_{\ell}\mathbf{u}(\tau+\varepsilon)
+\delta_{\ell-1}\gamma\mathbf{u}(\tau)+\Delta_{\ell}\mathbf{f}(\tau))$,
\end{center}
$t$ and $\tau$ being connected as before, which is well-determined. As a consequence, the formula (\ref{solsafe}) gives rise to
 one unique function $\mathbf{u}$ on $[t_0,t_f]$ such that $\Box\mathbf{u}=\mathbf{f}$ and this ends the proof.\end{proof}

\section{A matricial-based framework when $N$ is arbitrary}

Let us use a matricial-based framework for solving the functional equation for arbitrary $N$. 
In this section we assume throughout that $c_{N}$ and $\alpha$ are invertible, mainly for the ease of exposition. 
We shall say that the ($2N+2$)-tuple $(c_{-N},\ldots,c_0,\ldots,c_N,\alpha)$ is \textit{generic} in $(\mathbb{C}^{d\times d})^{2N+2}$ 
provided the foregoing system (\ref{sys1}) is Cramer. This system (\ref{sys1}) consists in $N$ equations w.r.t. $N$ unknowns, 
generalizing the auxiliary function $\varphi$ occuring in Section 5. Its coefficients are matricial polynomial functions w.r.t. 
$c_{N}^{-1}(c_0+\alpha)$ and the
 $2N-1$ matricial indeterminates $c_N^{-1}c_k$, $k\in\{-N,\ldots,N-1\}$ and $k\neq 0$, which in a sense generalize (\ref{rdl}). 
\begin{theorem}
If the ($2N+2$)-tuple $(c_{-N},\ldots,c_0,\ldots,c_N,\alpha)$ is generic, the functional equation (\ref{eq1}) admits one 
and only one solution.
\label{thm6-1}
\end{theorem}
Similarly, we shall say that the ($2N+1$)-tuple $(c_{-N},\ldots,c_0,\ldots,c_N)$ is
 \textit{generic} in $(\mathbb{C}^{d\times d})^{2N+1}$ provided the foregoing matrix $\Theta$ defined by (\ref{Theta}) is invertible. 
This matrix is associated to a linear system of $N$ equations w.r.t. $N$ unknowns $\mathbf{p}_1,\ldots\mathbf{p}_N$, 
which are the restrictions of $\mathbf{f}$ to parts of $[t_0-N\varepsilon,t_0[$ in order that
 $\mathbf{f}$ lies in the range of 
the operator $\Box$. 
\begin{theorem}
If the ($2N+1$)-tuple $(c_{-N},\ldots,c_0,\ldots,c_N)$ is generic, for all $\mathbf{g}:[t_0,t_f]\rightarrow \mathbb{C}^d$ 
there exists one and only one extension $\mathbf{f}:\mathbb{R}\rightarrow\mathbb{C}^d$ such that 
$Supp(\mathbf{f})\subset[t_0-N\varepsilon,t_f+N\varepsilon]$ and
 the functional equation $\mathbf{f}=\Box\mathbf{u}$ admits one and only one solution in $[t_0,t_f]$.
\label{thm6-2}
\end{theorem}
\begin{proof} (Of both theorems.)\\
Let $\mathbf{u}:\mathbb{R}\rightarrow\mathbb{C}^d$ be a solution to equation (\ref{eq1}), related to $\mathbf{A},\mathbf{B}$
 and $\mathbf{U}$ defined over $\mathbb{R}$ as follows:
\begin{equation}
\mathbf{B}(t)=\begin{pmatrix}
c_N^{-1}\mathbf{f}(t)\\0\\\vdots\\0\end{pmatrix}\in\mathbb{C}^{2dN},~~~~\mathbf{U}(t)=\begin{pmatrix}
\mathbf{u}(t+(N-1)\varepsilon)\\\vdots\\\mathbf{u}(t)\\\vdots\\\mathbf{u}(t-N\varepsilon)\end{pmatrix}\in\mathbb{C}^{2dN},
\label{data2}
\end{equation}
{\footnotesize\begin{equation}
\mathbf{A}(t)=\begin{pmatrix}-c_N^{-1}c_{N-1}\chi(t+(N-1)\varepsilon) & \ldots & 
-c_N^{-1}(\alpha+c_0\chi(t)) & \ldots & -c_N^{-1}c_{-N}\chi(t-N\varepsilon)\\
I_d & 0 & \ldots & \ldots & 0\\
0 & \ddots & \ddots & & \vdots\\
\vdots & \ddots & \ddots & \ddots & \vdots\\
0 & \ldots & 0 & I_d & 0
\end{pmatrix}.
\label{data3}
\end{equation}}
The matrix $\mathbf{A}(t)\in\mathbb{C}^{2dN\times 2dN}$ is locally constant w.r.t $t$. Especially, when 
$t\in[t_0+N\varepsilon,t_f-N\varepsilon]$, all the characteristic functions of the entries of $\mathbf{A}(t)$ are equal to 1 
and we shall agree to denote by $\mathbf{A}$ the constant value of $\mathbf{A}(t)$ in this segment. 
Let us mention that the matrix $\mathbf{A}$ depends only on the $2N+2$ matrices $c_i$ and $\alpha$ and 
is nothing but the block companion matrix of the functional equation (\ref{eq1}).

When $t<t_0-N\varepsilon$ or $t>t_f+N\varepsilon$, we have already seen that the solution to the equation (\ref{eq1}) is given 
by $\mathbf{u}(t)=\alpha^{-1}\mathbf{f}(t)$. Now, if $t\in[t_0-N\varepsilon,t_f-N\varepsilon]$, (\ref{eq2}) may be written as
\begin{eqnarray}
\mathbf{u}(t+N\varepsilon)=c_N^{-1}[\mathbf{f}(t)-(c_{N-1}\chi(t+(N-1)\varepsilon)\mathbf{u}(t+(N-1)\varepsilon)+\ldots\nonumber\\
+(\alpha+c_0\chi(t))\mathbf{u}(t)+\ldots+c_{-N}\chi(t-N\varepsilon)\mathbf{u}(t-N\varepsilon))].\hskip 2cm
\label{func3}
\end{eqnarray}
When $t\in[t_0-N\varepsilon,t_0[$, this equation expresses $\mathbf{u}(t+N\varepsilon)$ explicitly as a 
combination of the $N$ auxiliary functions $\varphi_k(\tau)=\mathbf{u}(\tau-k\varepsilon)$, 
$1\leq k\leq N$, which are at this level unknown. Obviously, the functions $\varphi_k(\tau)$, $1\leq k\leq N$, 
stand for the last $N$ components of $\mathbf{U}(\tau)$. 
By using definitions (\ref{data2}) and (\ref{data3}), recurrence (\ref{func3}) may be rewritten as 
\begin{center}
$\mathbf{U}(t+\varepsilon)=\mathbf{A}(t)\mathbf{U}(t)+\mathbf{B}(t)$, $\forall t\in[t_0-N\varepsilon,t_f-N\varepsilon]$.
\end{center} 
The preceding recurrence and Lemma \ref{lem1-sagv} yield the following formula
\begin{equation}
\mathbf{U}(t+n\varepsilon)=\mathbf{C}_n(t)\mathbf{U}(t)+\mathbf{F}_n(t)
\label{Ut}
\end{equation}
for $t_0-N\varepsilon\leq t\leq t+(n-1)\varepsilon\leq t_f-N\varepsilon$, where
\begin{equation}
\mathbf{C}_n(t)=\left\{\begin{array}{ll}
\displaystyle \coprod_{k=0}^{n-1} \mathbf{A}(t+k\varepsilon) & \mbox{if }t+(n-1)\varepsilon<t_0+N\varepsilon,\\
\displaystyle \mathbf{A}^{n-N}\left(\coprod_{k=0}^{N-1} \mathbf{A}(t+k\varepsilon)\right) & \mbox{otherwise}
\end{array}\right.
\label{sagv-gen3}
\end{equation}
and 
\begin{equation}
\mathbf{F}_n(t)=\sum_{k=0}^{n-1}\left(\coprod_{\ell=k+1}^{n-1}\mathbf{A}(t+\ell\varepsilon)\right)\mathbf{B}(t+k\varepsilon).
\label{Fn}
\end{equation} 
Focusing on the first row of (\ref{Ut}), i.e. the first component $\mathbf{u}(t+(n+N-1)\varepsilon)$ of $\mathbf{U}(t+n\varepsilon)$, 
we emphasize on the fact that the unknown function $\mathbf{u}(t)$ may be expressed as a linear combination of the delayed functions 
$\mathbf{f}(t+k\varepsilon)$, $k\in\mathbb{Z}$ through $\mathbf{F}_n(t)$, as well as the auxiliary functions 
$\varphi_k(\tau)$, $1\leq k\leq N$, through $\mathbf{U}(t)$. 
So, as a first conclusion of these calculations, $\mathbf{u}(t)$ is well-determined in the whole interval $]-\infty,t_f]$, as a 
combination of the auxiliary functions $\varphi_k$.\\
When $t\in[t_f-N\varepsilon,t_f+N\varepsilon]$, we may deduce from (\ref{eq1}) and (\ref{CMA}) the two systems
 of $N$ functional equations which are analogous to (\ref{bound3}) and (\ref{bound4}), i.e.
{\footnotesize \begin{equation}
\left\{\begin{array}{ll}
\mathbf{f}(t)=\alpha\mathbf{u}(t)+c_{-N}\mathbf{u}(t-N\varepsilon)+\ldots+c_{N-1}\mathbf{u}(t+(N-1)\varepsilon) & 
\mbox{ in }]t_f-N\varepsilon,t_f-(N-1)\varepsilon]\\
\mathbf{f}(t)=\alpha\mathbf{u}(t)+c_{-N}\mathbf{u}(t-N\varepsilon)+\ldots+c_{N-2}\mathbf{u}(t+(N-2)\varepsilon) &
 \mbox{in }]t_f-(N-1)\varepsilon,t_f-(N-2)\varepsilon]\\
~~\vdots & \\
\mathbf{f}(t)=\alpha\mathbf{u}(t)+c_{-N}\mathbf{u}(t-N\varepsilon)+\ldots+c_{0}\mathbf{u}(t) & \mbox{in }]t_f-\varepsilon,t_f]\\
\end{array}
\right.
\label{sys3}
\end{equation}}  
and
{\footnotesize \begin{equation}
\left\{\begin{array}{ll}
\mathbf{f}(t)=\alpha\mathbf{u}(t)+c_{-N}\mathbf{u}(t-N\varepsilon)+\ldots+c_{-1}\mathbf{u}(t-\varepsilon) & 
\mbox{ in }]t_f,t_f+\varepsilon]\\
\mathbf{f}(t)=\alpha\mathbf{u}(t)+c_{-N}\mathbf{u}(t-N\varepsilon)+\ldots+c_{-2}\mathbf{u}(t-2\varepsilon) &
 \mbox{in }]t_f+\varepsilon,t_f+2\varepsilon]\\
~~\vdots & \\
\mathbf{f}(t)=\alpha\mathbf{u}(t)+c_{-N}\mathbf{u}(t-N\varepsilon) & \mbox{in }]t_f+(N-1)\varepsilon,t_f+N\varepsilon]\\
\end{array}
\right..
\label{sys4}
\end{equation}}  

Let us consider the case when $\alpha$ is invertible. We convert (\ref{sys3}) into a linear system for the $N$ auxiliary 
unknown functions $\varphi_k$. To do this, we note first that (\ref{Ut}), with $n=M$ and $t=\tau-N\varepsilon$, writes as
\begin{equation}
\mathbf{U}(\tau+(M-N)\varepsilon)=\mathbf{C}_M(\tau-N\varepsilon)\mathbf{U}(\tau-N\varepsilon)+\mathbf{F}_M(\tau-N\varepsilon).
\label{recU}
\end{equation}
The first $N$ components of $\mathbf{U}(\tau-N\varepsilon)$ which occur in the right-hand side of the previous equality are
 the auxiliary functions while
the last $N$ components are values of the function $\alpha^{-1}\mathbf{f}$ at instants $\tau-(N+k)\varepsilon$, $1\leq k\leq N$. 
Next, by shifting adequately $t$ in the system of equations (\ref{sys3}), we get 
for all $t\in]t_f-N\varepsilon,t_f-(N-1)\varepsilon]$ the following equality
\begin{equation}
\begin{pmatrix}\mathbf{f}(t)\\\mathbf{f}(t+\varepsilon)\\\vdots\\\mathbf{f}(t+(N-1)\varepsilon)\end{pmatrix}=\mathbf{D}\begin{pmatrix}
\mathbf{u}(t+(N-1)\varepsilon)\\\vdots\\\mathbf{u}(t-N\varepsilon)
\end{pmatrix} 
\label{matrixS}
\end{equation}
where
\begin{center}
{\small $\mathbf{D}=\begin{pmatrix} c_{N-1} & \ldots & \ldots & c_1 & (\alpha+c_0) & c_{-1} & \ldots & c_{-N}\\
c_{N-2} & \ldots & c_1 & (\alpha+c_0) & c_{-1} & \ldots & c_{-N} & 0\\
\vdots & \addots & \addots & \addots & & \addots & \addots & \vdots\\
c_1 & \addots & \addots & & \addots & \addots & & \vdots\\
(\alpha+c_0) & c_{-1} & \ldots & c_{-N} & 0 & \ldots & \ldots & 0
\end{pmatrix}\in\mathbb{C}^{Nd\times(2Nd)}$}
\end{center}
is a rectangular block Hankel matrix. By setting $t=\tau+(M-N)\varepsilon$ in (\ref{matrixS}) and using (\ref{recU}), we easily get
\begin{equation}
\mathbf{D}\mathbf{C}_M(\tau-N\varepsilon)\mathbf{U}(\tau-N\varepsilon)=
\begin{pmatrix}\mathbf{f}(\tau+(M-N)\varepsilon)\\\vdots\\\mathbf{f}(\tau+(M-1)\varepsilon)\end{pmatrix}-
\mathbf{D}\mathbf{F}_M(\tau-N\varepsilon).
\label{sys1}
\end{equation}
By partitioning the $(Nd)\times(2Nd)$ matrix $\mathbf{D}\mathbf{C}_M(\tau-N\varepsilon)$ as $(\Theta_1~\Theta_2)$, where 
$\Theta_1,\Theta_2\in\mathbb{C}^{Nd\times Nd}$, we see that the previous linear system consists in $N$ vectorial equations in 
$\mathbb{C}^d$ depending on the $N$ auxiliary functions $\varphi_k$ and is Cramer provided that $\Theta_1$ is invertible. 
Under the assumption of genericity, the previous condition is satisfied and thus, there exists one and only one solution to 
(\ref{sys1}). The auxiliary functions being determined, as a rule, the function $\mathbf{u}$ 
is well defined and unique on $]-\infty,t_f]$. Since $\alpha$ is also invertible by assumption, we may compute the 
various extensions of $\mathbf{u}$ on $[t_f,t_f+N\varepsilon]$ using the system (\ref{sys4}). 
This ends the proof of Theorem 
\ref{thm6-1}.

The proof of the second theorem proceeds in the same way. 
Let $\mathbf{p}_k$ be the unknown restriction of $\mathbf{f}$ to
 $[t_0-k\varepsilon,t_0-(k-1)\varepsilon[$, for all $k\in\{1,\ldots,N\}$. 
Similarly let us define $\mathbf{q}_k$ as the unknown restriction of $\mathbf{f}$ to 
$]t_f+(k-1)\varepsilon,t_f+k\varepsilon]$, for all $k\in\{1,\ldots,N\}$. 
When $\alpha$ is zero and $\det(c_N)\neq 0$, the system (\ref{CMA}) shows that $\mathbf{u}(t)$ may be arbitrarily
 chosen outside $[t_0,t_f]$. 
Indeed, the auxiliary functions $\varphi_k$ do not occur in the calculation of $\mathbf{u}$. 
However, $\mathbf{u}$ is well-determined in the whole interval $[t_0,t_f]$ provided $\mathbf{g}$ is given in 
this interval and the $\mathbf{p}_k$, $1\leq k\leq N$, are already determined. 
In order that $\mathbf{u}$ satisfies the systems (\ref{sys3}) and (\ref{sys4}), the restrictions
 $\mathbf{p}_k$ and $\mathbf{q}_k$, $1\leq k\leq N$, must be chosen adequately. 
Let us focus on (\ref{sys3}) or, equivalently, on (\ref{sys1}). 
Indeed, both systems still hold when $\alpha=0$. 
However, our approach is different from previously since the unknowns which are the $\mathbf{p}_k$'s are located exclusively
 in the term $-\mathbf{D}\mathbf{F}_M(\tau-N\varepsilon)$. 
Indeed, the above discussion shows that $\mathbf{U}(\tau-N\varepsilon)$ does not depend on the restrictions $\mathbf{p}_k$ 
nor the matrix $(\mathbf{f}(\tau+(M-k)\varepsilon))_{1\leq k\leq n}$. 
Let us determine the coefficient $[\Theta]_{ij}\in\mathbb{C}^{d\times d}$ of $\mathbf{p}_j(\tau)$ in the $i$-th equation in 
the system (\ref{sys1}). To do this, we use the convention that for some matrix $\Theta\in\mathbb{C}^{(Nd)\times (Nd)}$, the symbol $[\Theta]_{ij}$
 denotes the $(d\times d)$-block located at the ($i,j$)-entry of $\Theta$. Then, help to formulas (\ref{data2}) and (\ref{Fn}), we see that
\begin{equation}
[\Theta]_{ij}=-\left[D\coprod_{k=N-j+1}^{M-1}\mathbf{A}(\tau+(k-N)\varepsilon)\right]_{i1}c_N^{-1}.
\label{Theta}
\end{equation}
We assume that the multiplet $(c_{-N},\ldots,c_N)$ is generic, which amounts to say that $\Theta$ is invertible. 
Therefore, (\ref{sys1}) is Cramer and admits one and only one solution $(\mathbf{p}_1,\ldots,\mathbf{p}_N)$. 
Next, by using the same transformation which led to (\ref{matrixS}), the system (\ref{sys4}) may be rewritten as
\begin{center}
{\small $\begin{pmatrix}\mathbf{f}(t)\\\vdots\\\mathbf{f}(t+(N-1)\varepsilon)\end{pmatrix}=
\begin{pmatrix}\mathbf{q}_1(t)\\\vdots\\\mathbf{q}_N(t)\end{pmatrix}=\begin{pmatrix}c_{-1} & c_{-2} & \ldots & c_{-N}\\ 
c_{-2} & \ldots & c_{-N} & 0\\ \vdots & \addots & \addots & \vdots\\c_{-N} & 0 & \ldots & 0\end{pmatrix}
\begin{pmatrix}\mathbf{u}(t-\varepsilon)\\\vdots\\\mathbf{u}(t-N\varepsilon)\end{pmatrix}$},
\end{center}
for all $t\in]t_f,t_f+\varepsilon]$. 
Since the matrix $(\mathbf{u}(t-k\varepsilon))_{1\leq k\leq N}$ is known, we obtain as a conclusion that the restrictions
 $(\mathbf{p}_k,\mathbf{q}_k)$, $1\leq k\leq N$, are uniquely determined by $\mathbf{g}$, which ends the proof of Theorem \ref{thm6-2}.
\end{proof}

\section{Final comments}

Before concluding this paper, it is natural to ask if there exist other ways to handle the functional equation 
$\alpha\mathbf{u}+\mathcal{C}\star(\chi\mathbf{u})=\mathbf{f}$. 
Since (\ref{eq1}) is a convolution equation, one might want to use the discrete Laplace (or $Z$-) transform to solve it. 
Indeed, if the $Z$-transform is defined by 
$\mathcal{L}\mathbf{u}(p)=\sum_{t\in \varepsilon \mathbb{Z}}\exp (-pt)\mathbf{u}(t)$ then one has the formula
 $\mathcal{L}(\mathbf{u}\star \mathbf{v})(p)=(\mathcal{L}\mathbf{u})(\mathcal{L}\mathbf{v})$ for all 
functions $\mathbf{u}$ and $\mathbf{v}$ on the real line, with compact support and values in suitable algebras. 
So, when $\alpha=0$, the image 
of (\ref{eq1}) is $\mathcal{L}(\mathcal{C})\mathcal{L}(\chi\mathbf{u})=\mathcal{L}(f)$ and, thus, 
if $\mathcal{L}(\mathcal{C})$ is invertible, we get
\begin{center}
$\mathcal{L}(\chi \mathbf{u})(p)=\left( \sum_{k\in \mathbb{Z}}c_k\exp
(pk\varepsilon )\right) ^{-1}\left( \sum_{k\in \mathbb{Z}}\exp (-pk\varepsilon )%
\mathbf{f}(k\varepsilon )\right) $.
\end{center}
Unfortunately, the determination of the original $\mathbf{u}$ has been transferred to a difficult problem of moments. 
Moreover, this approach does not apply to the general case $\alpha \neq 0$.\\
Another way to proceed would be to apply the Picard's method. 
Indeed, one sees that (\ref{eq1}) is
equivalent to
\begin{center}
$\mathbf{u}=\alpha ^{-1}\mathbf{f}-\alpha ^{-1}\Box \mathbf{u}=\alpha ^{-1}%
\mathbf{f}-\alpha ^{-1}\Box (\alpha ^{-1}\mathbf{f}-\alpha ^{-1}\Box \mathbf{%
u})=\ldots$\\
$=\alpha ^{-1}\mathbf{f}+\sum_{k=1}^\infty (-1)^k\alpha ^{-k-1}\Box ^k%
\mathbf{f}$.
\end{center}
Thus if $\Vert \alpha ^{-1}\Vert $ is small enough, one may hope that $\Box $ is a contraction
 of a suitable Banach or Fr\'echet function space $\mathcal{F}$. 
The main difficulty in this approach is the requirement that the vector
space $\mathcal{F}$ is complete, and it is related to the increasing number
of points of discontinuity of $\Box ^k\mathbf{f}$, as $k$ increases.\\
By comparison, the approach developed in this paper is less complicated and more robust in the sense that the solution is
 given explicitly. 
Indeed, the central parts of the proofs of the four previous theorems highlight not only the well-posedness of (\ref{eq1})
 but also an effective computational way to construct interval by interval the solution $\mathbf{u}$.

In scope of future work, we suggest the three following problems. 
First, the formulas (\ref{Ut}) and (\ref{sagv-gen3}) imply that, in the very special case when $\mathbf{f}$ is constant, 
$\mathbf{u}(t)$ may be expressed through exponential-monomial functions inside $[t_0+N\varepsilon,t_f-N\varepsilon]$. 
So, in this case, the search for pseudo-periodic solutions of (\ref{eq1}) is linked to the difficult elimination problem
 of requiring that the spectrum $Sp(\mathbf{A})$ of $\mathbf{A}$ is included in the unit circle of $\mathbb{C}$. 
Second, we conjecture the following variant of Theorem \ref{ilpmm}. 
If the coefficients $\kappa_{m_1,m_2,\ldots}$ occuring in formula (\ref{rdl}) are rational or even algebraic over 
$\mathbb{Q}$ and, in contrast, the entries $\zeta_{i,j}$ and $\xi_{i,j}$ generate a subfield of $\mathbb{C}$ of
 transcendance degree equal to $2d^2$, 
then all the quantities $\psi_{i,j}^{(k)}(\zeta,\xi)$ and $\det(\psi^{(k)})(\zeta,\xi)$ do not vanish. 
Lastly, we conjecture also that for all $N\geq 2$, the genericity assumption on the multiplets in Theorems \ref{thm6-1} 
and \ref{thm6-2} hold on open dense subsets of appropriate products of matricial algebras. 
This conjecture seems natural in light of the two theorems presented in Section 5.



\end{document}